\documentclass[12pt]{amsart}

\usepackage{amssymb}

\usepackage{enumerate}

\usepackage{graphicx}

\makeatletter
\@namedef{subjclassname@2010}{%
  \textup{2010} Mathematics Subject Classification}
\makeatother

\newtheorem{thm}{Theorem}[section]

\newtheorem{lemma}[thm]{Lemma}

\theoremstyle{definition}

\def\N{\mathbb N}

\def\pmod #1{\ ({\rm mod}\ #1)}
\def\le{\leqslant}
\def\ge{\geqslant}

\numberwithin{equation}{section}

\begin{document}


\baselineskip=17pt



\title[Exceptional sets: two squares and $s$ biquadrates]{Exceptional sets in Waring's problem: two squares and $s$ biquadrates}

\author[L. Zhao]{Lilu  Zhao}
\address{School of Mathematics, Hefei University of Technology, Heifei 230009, People's Republic of China}
\email{zhaolilu@gmail.com}

\date{}

\begin{abstract}
Let $R_s(n)$ denote the number of representations of the positive
number $n$ as the sum of two squares and $s$ biquadrates. When
$s=3$ or $4$, it is established that the anticipated asymptotic
formula for $R_s(n)$ holds for all $n\le X$ with at most
$O(X^{(9-2s)/8+\varepsilon})$ exceptions.
\end{abstract}



\maketitle
\section{Introduction}

Waring's problem for sums of mixed powers involving one or two
squares has been widely investigated. In 1987-1988, Br\"{u}dern
\cite{B1,B2} considered the representation of $n$ in the form
\begin{align*}n=x_1^2+x_2^2+y_1^{k_1}+\cdots+ y_s^{k_s},\end{align*}
with $k_1^{-1}+\cdots +k_s^{-1}>1$. Linnik \cite{L} and Hooley
\cite{H} investigated sums of two squares and three cubes. In
2002, Wooley \cite{Wooley02Q} investigated the exceptional set
related to the asymptotic formula in Waring's problem involving
one square and five cubes. Recently, Br\"{u}dern and Kawada
\cite{BK} established the asymptotic formula for the number of
representations of the positive number $n$ as the sum of one
square and seventeen fifth powers.

Let $R_s(n)$ denote the number of representations of the positive
number $n$ as the sum of two squares and $s$ biquadrates. Very
recently, subject to the truth of the Generalised Riemann
Hypothesis and the Elliott-Halberstam Conjecture, Friedlander and
Wooley \cite{FW} established that $R_3(n)>0$ for all large $n$
under certain congruence conditions. They also showed that if one
is prepared to permit a small exceptional set of natural numbers
$n$, then the anticipated asymptotic formula for $R_s(n)$ can be
obtained. To state their results precisely, we introduce some
notations. We define
\begin{align}\mathfrak{S}_s(n)=\sum_{q=1}^\infty\sum_{\substack{a=1\\ (a,q)=1}}^{q}q^{-2-s}S_2(q,a)^2S_4(q,a)^se(-na/q),\end{align}
where the Gauss sum $S_k(q,a)$ is defined as
\begin{align}\label{gauss}S_k(q,a)=\sum_{r=1}^{q}e(ar^k/q).\end{align}
As in \cite{FW}, we refer a function $\psi(t)$ as being a
\textit{sedately increasing function} when $\psi(t)$ is a function
of a positive variable $t$, increasing monotonically to infinity,
and satisfying the condition that when $t$ is large, one has
$\psi(t)=O(t^\delta)$ for a positive number $\delta$ sufficiently
small in the ambient context. Then we introduce $E_s(X,\psi)$ to
denote the number of integers $n$ with $1\le n\le X$ such
that\begin{align}\label{lower}\big|R_s(n)-c_s
\Gamma(\frac{5}{4})^4\mathfrak{S}_s(n)n^{s/4}\big|>n^{s/4}\psi(n)^{-1},\end{align}
where $c_3=\frac{2}{3}\sqrt{2}$ and $c_4=\frac{1}{4}\pi$.
Friedlander and Wooley \cite{FW} established the upper bounds
\begin{align}\label{E3}E_3(X,\psi)\ll X^{1/2+\varepsilon}\psi(X)^2\end{align}
and
\begin{align}\label{E4}E_4(X,\psi)\ll X^{1/4+\varepsilon}\psi(X)^4,\end{align}
where $\varepsilon>0$ is arbitrary small.

The main purpose of this note is to establish the following
result.
\begin{thm}\label{theorem}
Suppose that $\psi(t)$ is a sedately increasing function. Let
$E_s(X,\psi)$ be defined as above. Then for each $\varepsilon>0$,
one has\begin{align}E_3(X,\psi)\ll
X^{3/8+\varepsilon}\psi(X)^2\end{align} and
\begin{align}E_4(X,\psi)\ll X^{1/8+\varepsilon}\psi(X)^2,\end{align}
where the implicit constants may depend on $\varepsilon$.
\end{thm}

%

We establish Theorem \ref{theorem} by means of the
Hardy-Littlewood method. In order to estimate the corresponding
exceptional sets effectively, we employ the method developed by
Wooley \cite{Wooley02,Wooley02Q}.

As usual, we write $e(z)$ for $e^{2\pi iz}$. Whenever
$\varepsilon$ appears in a statement, either implicitly or
explicitly, we assert that the statement holds for each
$\varepsilon>0$. Note that the "value" of $\varepsilon$ may
consequently change from statement to statement. We assume that
$X$ is a large positive number, and $\psi(t)$ is a sedately
increasing function.

 \vskip3mm

\section{Preparations}

Throughout this section, we assume that $X/2< n\le X$. For $k\in
\{2,4\}$, we define the exponential
sum\begin{align*}f_k(\alpha)=\sum_{1\le x\le P_k}e(\alpha
x^k),\end{align*}where $P_k=X^{1/k}$. We take $s$ to be either $3$
or $4$. By orthogonality, we
have\begin{align}\label{orth}R_s(n)=\int_0^1f_2(\alpha)^2f_4(\alpha)^se(-n\alpha)d\alpha.\end{align}
When $Q$ is a positive number, we define $\mathfrak{M}(Q)$ to be
the union of the intervals
\begin{align*}\mathfrak{M}_{Q}(q,a)=\{\alpha:\ |q\alpha-a|\le QX^{-1}\},\end{align*}
with $1\le a\le q\le Q$ and $(a,q)=1$. Whenever $Q\le X^{1/2}/2$,
the intervals $\mathfrak{M}_{Q}(q,a)$ are pairwise disjoint for
$1\le a\le q\le Q$ and $(a,q)=1$. Let $\nu$ be a sufficiently
small positive number, and let $R=P_4^{\nu}$. We take
$\mathfrak{M}=\mathfrak{M}(R)$ and
$\mathfrak{m}=(R/N,1+R/N]\setminus \mathfrak{M}$. Write
$$v_k(\beta)=\int_{0}^{P_k}e(\gamma^k\beta)d\gamma.$$
One has the estimate $$v_k(\beta) \ll P_k(1+X|\beta|)^{-1/k}.$$
For $\alpha\in
\mathfrak{M}_{X^{1/2}/2}(q,a)\subseteq\mathfrak{M}(X^{1/2}/2)$, we
define\begin{align}\label{fkast}f_k^{\ast}(\alpha)=q^{-1}S_k(q,a)v_k(\alpha-a/q).\end{align}
It follows from Theorem 4.1 \cite{V} that whenever $\alpha\in
\mathfrak{M}_{X^{1/2}/2}(q,a)$, one has
\begin{align}\label{difference}f_k(\alpha)-f_k^{\ast}(\alpha)\ll q^{1/2}(1+X|\alpha-a/q|)^{1/2}X^{\varepsilon}.\end{align}
We define the multiplicative function $w_k(q)$ by taking
\begin{align*}w_k(p^{uk+v})=\begin{cases}kp^{-u-1/2}, &\textrm{ when } u\ge 0 \textrm{ and } v=1,
\\p^{-u-1}, &\textrm{ when } u\ge 0 \textrm{ and } 2\le v\le k.\end{cases}\end{align*}
Note that $q^{-1/2}\le w_k(q) \ll q^{-1/k}$. Whenever $(a,q)=1$,
we have
\begin{align*}q^{-1}S_k(q,a)\ll w_k(q).\end{align*}
Therefore for $\alpha=a/q+\beta\in
\mathfrak{M}_{X^{1/2}/2}(q,a)\subseteq\mathfrak{M}(X^{1/2}/2)$,
one has\begin{align}\label{f-ast}f_k^{\ast}(\alpha)\ll
w_k(q)P_k(1+X|\beta|)^{-1/k} \ll
P_kq^{-1/k}(1+X|\beta|)^{-1/k}.\end{align}

The following conclusion is (4.1) in \cite{FW}.
\begin{lemma}\label{lemmaM}
One has
\begin{align*}\int_{\mathfrak{M}}f_2(\alpha)^2f_4(\alpha)^se(-n\alpha)d\alpha=
c_s
\Gamma(\frac{5}{4})^4\mathfrak{S}_s(n)n^{s/4}+O(n^{s/4-\kappa+\varepsilon}),\end{align*}
for a suitably small positive number $\kappa$.
\end{lemma}

The next result provides the value of the Gauss sum $S_2(q,a)$.
\begin{lemma}\label{lemma22}
The Gauss sum $S_2(q,a)$ satisfies the following properties.

(i) If $(2a,q)=1$, then
\begin{align*}S_2(q,a)=\Big(\frac{a}{q}\Big)S_2(q,1).\end{align*}
Here by $\Big(\frac{a}{q}\Big)$ we denote the Jacobi symbol.

(ii) If $q$ is odd,
then\begin{align*}S_2(q,1)=\begin{cases}q^{1/2},
 &\ \textrm{ if }\ q\equiv1\pmod{4},
 \\ iq^{1/2},
 &\ \textrm{ if }\ q\equiv3\pmod{4}.\end{cases}\end{align*}

(iii) If $(2,a)=1$, then\begin{align*}S_2(2^m,a)=\begin{cases}0,
 &\ \textrm{ if }\ m=1,
 \\ 2^{m/2}(1+i^a),
 &\ \textrm{ if }\ m \textrm{ is even},
  \\ 2^{(m+1)/2}e(a/8),
 &\ \textrm{ if }\ m>1 \textrm{ and } m \textrm{ is odd}.\end{cases}\end{align*}

(iv) If $(q_1,q_2)=1$, then
\begin{align*}S_2(q_1q_2,a_1q_2+a_2q_1)=S_2(q_1,a_1)S_2(q_2,a_2).\end{align*}
\end{lemma}
\begin{proof}These properties can be found in Lemma 2 \cite{HT}.\end{proof}

 \vskip3mm

\section{The Proof of Theorem \ref{theorem}}
Let $\tau$ be a fixed sufficiently small positive number. Set
$Y=P_4^{3/2+\tau}\psi(X)^2$. We define
$\mathfrak{m}_1=\mathfrak{m}\setminus \mathfrak{M}(X^{1/2}/2)$,
$\mathfrak{m}_2=\mathfrak{M}(X^{1/2}/2)\setminus \mathfrak{M}(Y)$,
$\mathfrak{m}_3=\mathfrak{M}(Y)\setminus \mathfrak{M}(P_4)$ and
$\mathfrak{m}_4=\mathfrak{M}(P_4)\setminus \mathfrak{M}$. Let
$\eta(n)$ be sequence of complex numbers satisfying $|\eta(n)|=1$.
Let $\mathcal{Z}$ be a subset of $\{n\in \N:\ X/2< n\le X\}$. We
abbreviate card($\mathcal{Z}$) to $Z$. Then we introduce the
exponential sum $\mathcal{E}(\alpha)$ by
\begin{align*}\mathcal{E}(\alpha)=\sum_{n\in \mathcal{Z}}
\eta(n)e(-n\alpha).\end{align*}For $1\le j\le 4$, we define
\begin{align}\label{Ij}\mathcal{I}_j=
\int_{\mathfrak{m}_j}\big|f_2(\alpha)^2f_4(\alpha)^s\mathcal{E}(\alpha)\big|d\alpha.\end{align}

\begin{lemma}\label{lemma31}Let $\mathcal{I}_1$ be defined in
(\ref{Ij}). Then we have
\begin{align*}
\mathcal{I}_1\ll
P_4^{4-\frac{1}{4}+\frac{s-3}{2}+\varepsilon}Z^{1/2}+P_4^{s-\frac{1}{4}+\varepsilon}Z.\end{align*}
\end{lemma}\begin{proof}
For any $\alpha\in \mathfrak{m}_1$, there exist $a$ and $q$ with
$1\le a\le q\le 2X^{1/2}$ and $(a,q)=1$ such that $|q\alpha-a|\le
X^{-1/2}/2$. Since $\alpha\in \mathfrak{m}_1$, we conclude that
$q>X^{1/2}/2$. It follows from Weyl's inequality (Lemma 2.4
\cite{V}) that
$$f_2(\alpha)\ll P_2^{1/2+\varepsilon}\ \textrm{ for }\ \alpha\in \mathfrak{m}_1.$$
Thus we have \begin{align*}\mathcal{I}_1 \ll & P_2^{1+\varepsilon}
\int_{\mathfrak{m}_1}\big|f_4(\alpha)^s\mathcal{E}(\alpha)\big|d\alpha
\\ \ll & P_2^{1+\varepsilon}
\big(\int_{0}^1\big|f_4(\alpha)^6\big|d\alpha\big)^{\frac{1}{2}}
\big(\int_{0}^1\big|f_4(\alpha)^{2(s-3)}\mathcal{E}(\alpha)^2\big|d\alpha\big)^{\frac{1}{2}}.\end{align*}
By Hua's inequality (Lemma 2.5 \cite{V}) and Schwartz's
inequality,
$$\int_{0}^1\big|f_4(\alpha)^6\big|d\alpha\ll \Big(\int_{0}^1\big|f_4(\alpha)^4\big|d\alpha\Big)^{1/2}
\Big(\int_{0}^1\big|f_4(\alpha)^8\big|d\alpha\Big)^{1/2}\ll
P_4^{7/2+\varepsilon}.$$ When $s=4$, one has the bound
$\int_{0}^1\big|f_4(\alpha)^{2(s-3)}\mathcal{E}(\alpha)^2\big|d\alpha\ll
P_4Z+P_4^{\varepsilon}Z^2$. Then we can conclude that
\begin{align}\label{I1}
\mathcal{I}_1\ll P_4^{4-\frac{1}{4}+\frac{s-3}{2}+\varepsilon}Z^{1/2}+P_4^{s-\frac{1}{4}+\varepsilon}Z.\end{align}
Indeed when $s=3$, the estimate (\ref{I1}) holds with
$P_4^{s-\frac{1}{4}+\varepsilon}Z$ omitted.\end{proof}

\begin{lemma}\label{lemma32}Let $\mathcal{I}_2$ be given by
(\ref{Ij}). Then one has
\begin{align*}\mathcal{J}_2\ \ll P_4^{4-\frac{1}{4}+\frac{s-3}{2}+\varepsilon}Z^{1/2}
+P_4^{s-\tau/2+\varepsilon}\psi(X)^{-1}Z.\end{align*}
\end{lemma}\begin{proof}We introduce
$$\mathcal{J}_1=
\int_{\mathfrak{m}_2}\big|\big(f_2(\alpha)-f_2^\ast(\alpha)\big)^2f_4(\alpha)^s\mathcal{E}(\alpha)\big|d\alpha$$
and
$$\mathcal{J}_2=
\int_{\mathfrak{m}_2}\big|f_2^\ast(\alpha)^2f_4(\alpha)^s\mathcal{E}(\alpha)\big|d\alpha.$$
Note that $|f_2(\alpha)|^2\ll |f_2(\alpha)-f_2^\ast(\alpha)|^2+
|f_2^\ast(\alpha)|^2$, where $f_2^\ast(\alpha)$ is defined in
(\ref{fkast}). Then one has\begin{align}\label{J12}\mathcal{I}_2\
\ll\ \mathcal{J}_1+\mathcal{J}_2.\end{align} In view of
(\ref{difference}), we know $f_2(\alpha)-f_2^\ast(\alpha)\ll
P_2^{1/2+\varepsilon}$ for $\alpha\in\mathfrak{m}_2$. The argument
leading to (\ref{I1}) also implies
\begin{align}\label{J1}\mathcal{J}_1\ll P_4^{4-\frac{1}{4}+\frac{s-3}{2}+\varepsilon}Z^{1/2}+P_4^{s-\frac{1}{4}+\varepsilon}Z.\end{align}
One has, by Schwartz's inequality, that
\begin{align*}\mathcal{J}_2\le&\ \Big(
\int_{\mathfrak{m}_2}\big|f_4(\alpha)^6\big|d\alpha\Big)^{1/2}\mathcal{J}^{1/2}
\ll\ P_4^{7/4+\varepsilon}\mathcal{J}^{1/2},\end{align*} where
$\mathcal{J}$ is defined
as\begin{align*}\mathcal{J}=\int_{\mathfrak{m}_2}\big|f_2^\ast(\alpha)^4f_4(\alpha)^{2(s-3)}\mathcal{E}(\alpha)^2\big|d\alpha
.\end{align*} In order to handle $\mathcal{J}$, we need the
following estimate
\begin{align}\label{diandof}\int_{\mathfrak{m}_2}\big|f_2^\ast(\alpha)^4\big|e(-h\alpha)d\alpha
=\begin{cases}O(P_4^{4+\varepsilon}Y^{-1}), &\textrm{ when }
0<|h|\le 2X,
\\ O(P_4^{4+\varepsilon}), &\textrm{ when } h=0.\end{cases}\end{align}
Recalling the definition of $f_2^\ast(\alpha)$, we conclude that
\begin{align*}&\int_{\mathfrak{m}_2}\big|f_2^\ast(\alpha)^4\big|e(-h\alpha)d\alpha
\\=&{\sum}^\ast_{q\le X^{1/2}/2}{\int}^\ast_{|\beta|\le
\frac{1}{2qX^{1/2}}} q^{-4}\Big(\sum_{\substack{a=1\\
(a,q)=1}}^{q}|S_2(q,a)|^4e(-ha/q)\Big)|v_2(\beta)|^4e(-h\beta)d\beta,\end{align*}
where the notations ${\sum}^\ast$ and ${\int}^\ast$ mean either
$q>Y$ or $Xq|\beta|>Y$. Whenever $(a,q)=1$, one has by Lemma
\ref{lemma22} that $$|S_2(q,a)|=|S_2(q,1)|\le (2q)^{1/2}.$$ We
obtain
\begin{align*}
\Big|\sum_{\substack{a=1\\ (a,q)=1}}^{q}|S_2(q,a)|^4e(-ha/q)\Big|=
& |S_2(q,1)|^4 \Big|\sum_{\substack{a=1\\ (a,q)=1}}^{q}
e(-ha/q)\Big|
\\ \le & 4q^2 \Big|\sum_{\substack{a=1\\ (a,q)=1}}^{q}
e(-ha/q)\Big|\le 4q^2(q,h),\end{align*}whence
\begin{align*}\int_{\mathfrak{m}_2}\big|f_2^\ast(\alpha)^4\big|e(-h\alpha)d\alpha
\ll P_2^4{\sum}^\ast_{q\le X^{1/2}/2}{\int}^\ast_{|\beta|\le
\frac{1}{2qX^{1/2}}}
\frac{q^{-2}(q,h)}{(1+X|\beta|)^{2}}\,d\beta.\end{align*} When
$h=0$, we have
\begin{align*}\int_{\mathfrak{m}_2}\big|f_2^\ast(\alpha)^4\big|e(-h\alpha)d\alpha
\ll & P_2^4\sum_{q\le X^{1/2}/2}\int_{|\beta|\le
\frac{1}{2qX^{1/2}}} q^{-1}(1+X|\beta|)^{-2}d\beta
\\ \ll &
P_2^{4}X^{-1}\log X.\end{align*} When $h\not=0$, we get
\begin{align*}\int_{\mathfrak{m}_2}\big|f_2^\ast(\alpha)^4\big|e(-h\alpha)d\alpha
\ll &\ P_2^4Y^{-1}\sum_{q\le X^{1/2}/2}\int_{|\beta|\le
\frac{1}{2qX^{1/2}}}\frac{ q^{-1}(q,h)}{1+X|\beta|}\,d\beta
\\ \ll &\ P_2^{4}Y^{-1}X^{-1}(\log X)\sum_{q\le X^{1/2}/2}q^{-1}(q,h)
\\ \ll &\ P_2^{4}Y^{-1}X^{-1+\varepsilon}.\end{align*}
The conclusion (\ref{diandof}) is established. Now we are able to
estimate $\mathcal{J}$. When $s=4$,
\begin{align*}\mathcal{J}=\sum_{\substack{1\le x_1,x_2\le P_4 \\ n_1,n_2\in
\mathcal{Z}}}\eta(n_1)\overline{\eta(n_2)}
\int_{\mathfrak{m}_2}\big|f_2^\ast(\alpha)^4\big|e(-(x_1^4-x_2^4+n_1-n_2)\alpha)d\alpha.\end{align*}
On applying (\ref{diandof}), we can deduce that
\begin{align*}\mathcal{J}\ll& \sum_{\substack{1\le x_1,x_2\le P_4,\ n_1,n_2\in
\mathcal{Z}\\
x_1^4-x_2^4+n_1-n_2\not=0}}P_4^{4+\varepsilon}Y^{-1}+\sum_{\substack{1\le
x_1,x_2\le P_4,\
n_1,n_2\in \mathcal{Z}\\
x_1^4-x_2^4+n_1-n_2=0}}P_4^{4+\varepsilon}
\\ \ll &\ P_4^{6+\varepsilon}Z^2Y^{-1}+P_4^{4+\varepsilon}Z^2+P_4^{5+\varepsilon}Z.\end{align*}
Substituting $Y=P_4^{3/2+\tau}\psi(X)^2$, we finally obtain
$$\mathcal{J}\ll
P_4^{4+1/2-\tau+\varepsilon}\psi(X)^{-2}Z^2+P_4^{5+\varepsilon}Z,$$
whence
\begin{align*}\mathcal{J}_2\ \ll\
P_4^{4-\tau/2+\varepsilon}\psi(X)^{-1}Z+P_4^{4+1/4+\varepsilon}Z^{1/2}.\end{align*}
Similarly, when $s=3$, one has
$$\mathcal{J}\ll
P_4^{5/2-\tau+\varepsilon}\psi(X)^{-2}Z^2+P_4^{4+\varepsilon}Z$$
whence
\begin{align*}\mathcal{J}_2\ \ll\
P_4^{3-\tau/2+\varepsilon}\psi(X)^{-1}Z+P_4^{4-1/4+\varepsilon}Z^{1/2}.\end{align*}
Therefore, we conclude that
\begin{align}\label{J2}\mathcal{J}_2\ \ll P_4^{4-\frac{1}{4}+\frac{s-3}{2}+\varepsilon}Z^{1/2}
+P_4^{s-\tau/2+\varepsilon}\psi(X)^{-1}Z.\end{align} Combining
(\ref{J12}), (\ref{J1}) and (\ref{J2}), we conclude that
\begin{align}\label{I2}\mathcal{I}_2\ \ll\
P_4^{4-\frac{1}{4}+\frac{s-3}{2}+\varepsilon}Z^{1/2}+P_4^{s-\tau/2+\varepsilon}\psi(X)^{-1}Z.\end{align}
We complete the proof.\end{proof}

\begin{lemma}\label{lemma33}Let $\mathcal{I}_3$ be defined in
(\ref{Ij}). Then we have
\begin{align*}\mathcal{I}_3 \ll P_4^{4-\frac{1}{4}+\frac{s-3}{2}+\varepsilon}Z^{1/2}+P_4^{s-\tau+\varepsilon}\psi(X)^{-1}Z
.\end{align*}
\end{lemma}\begin{proof}
Similarly to (\ref{J12}) and (\ref{J1}), we can derive that
\begin{align}\label{I3K}\mathcal{I}_3 \ll P_4^{4-\frac{1}{4}+\frac{s-3}{2}+\varepsilon}Z^{1/2}+P_4^{s-\frac{1}{4}+\varepsilon}Z
+\mathcal{K},\end{align} where \begin{align*}\mathcal{K} =
\int_{\mathfrak{m}_3}\big|f_2^{\ast}(\alpha)^{2}f_4(\alpha)^s\mathcal{E}(\alpha)\big|d\alpha.\end{align*}
One has
\begin{align*}\mathcal{K} \le \sup_{\alpha\in
\mathfrak{m}_3}|f_4(\alpha)|&\Big(
\int_{\mathfrak{m}_3}\big|f_2^{\ast}(\alpha)^2f_4(\alpha)^4\big|d\alpha\Big)^{1/2}
\\ \times&\Big(
\int_{\mathfrak{m}_3}\big|f_2^{\ast}(\alpha)^2f_4(\alpha)^{2(s-3)}\mathcal{E}(\alpha)^2\big|d\alpha\Big)^{1/2}
.\end{align*} In view of (\ref{difference}) and (\ref{f-ast}), we
have for $\alpha\in \mathfrak{m}_3$ that
\begin{align*}f_4(\alpha) \ll P_4q^{-1/4}(1+X|\alpha-a/q|)^{-1/4}+Y^{1/2}X^{\varepsilon}\ll
P_4^{3/4+\tau/2+\varepsilon}\psi(X).\end{align*} Since
$f_2^{\ast}(\alpha)-f_2(\alpha)\ll P_2^{1/2}$ for $\alpha\in
\mathfrak{m}_3$, we easily deduce that
\begin{align*}
\int_{\mathfrak{m}_3}\Big|f_2^{\ast}(\alpha)^2f_4(\alpha)^4\big|d\alpha
\ll &\ P_2^{1/2}
\int_{0}^1\Big|f_2(\alpha)f_4(\alpha)^4\big|d\alpha+\int_{0}^1\big|f_2(\alpha)^2f_4(\alpha)^4\big|d\alpha
\\ \ll &\ P_4^{4+\varepsilon}.\end{align*}
Therefore we arrive at
\begin{align*}\mathcal{K} \ll P_4^{11/4+\tau/2+\varepsilon}\psi(X)\Big(
\int_{\mathfrak{m}_3}\big|f_2^{\ast}(\alpha)^2f_4(\alpha)^{2(s-3)}\mathcal{E}(\alpha)^2\big|d\alpha\Big)^{1/2}.\end{align*}
Similarly to (\ref{diandof}), we have the following estimate
\begin{align}\label{di-off}\int_{\mathfrak{M}(Y)}\big|f_2^\ast(\alpha)^2\big|e(-h\alpha)d\alpha
=\begin{cases}O(P_4^{\varepsilon}), &\textrm{ when } 0<|h|\le 2X,
\\ O(P_4^{\varepsilon}Y), &\textrm{ when } h=0.\end{cases}\end{align}
Note that
\begin{align*}&\ \int_{\mathfrak{M}(Y)}\big|f_2^\ast(\alpha)^2\big|e(-h\alpha)d\alpha
\\ =&\ \sum_{q\le Y}\int_{|\beta|\le
\frac{Y}{qX}} q^{-2}\Big(\sum_{\substack{a=1\\
(a,q)=1}}^{q}|S_2(q,a)|^2e(-ha/q)\Big)|v_2(\beta)|^2e(-h\beta)d\beta
\\ \ll&\ P_2^2\sum_{q\le Y}\int_{|\beta|\le
\frac{Y}{qX}} q^{-1}(q,h)(1+X|\beta|)^{-1}d\beta \\
\ll &\ (\log X)\sum_{q\le Y} q^{-1}(q,h).\end{align*} The desired
estimate (\ref{di-off}) follows easily from above. For $s=4$, we
derive that
\begin{align*}&
\int_{\mathfrak{m}_3}\big|f_2^{\ast}(\alpha)^2f_4(\alpha)^2\mathcal{E}(\alpha)^2\big|d\alpha
\\ \le &
\int_{\mathfrak{M}(Y)}\big|f_2^{\ast}(\alpha)^2f_4(\alpha)^2\mathcal{E}(\alpha)^2\big|d\alpha\\=&
\sum_{\substack{n_1,n_2\in \mathcal{Z}\\ 1\le x_1,x_2\le
P_4}}\eta(n_1)\overline{\eta(n_2)}\int_{\mathfrak{M}(Y)}\big|f_2^\ast(\alpha)^2\big|e(-(n_1-n_2+x_1^4-x_2^4)\alpha)d\alpha
\\ \ll &\ P_4^{2+\varepsilon}Z^2+P_4^{\varepsilon}Y(P_4^{\varepsilon}Z^2+P_4Z)
\\ \ll &\
\big(P_4^{2+\varepsilon}+P_4^{3/2+\tau+\varepsilon}\psi(X)^{2}\big)Z^2+P_4^{5/2+\tau+\varepsilon}\psi(X)^{2}Z,\end{align*}
whence
\begin{align*}\mathcal{K} \ll \big(P_4^{15/4+\tau/2+\varepsilon}\psi(X)+P_4^{7/2+\tau+\varepsilon}\psi(X)^2
\big)Z+P_4^{4+\tau+\varepsilon}\psi(X)^{2}Z^{1/2}.\end{align*} In
particular, we have
\begin{align*}\mathcal{K} \ll
P_4^{4+\frac{1}{4}+\varepsilon}Z^{1/2}+P_4^{4-\tau+\varepsilon}\psi(X)^{-1}Z\end{align*}
provided that $\psi(X)\ll X^{1/64-\tau}$. For $s=3$, by
(\ref{di-off}) we have
\begin{align*}
\int_{\mathfrak{m}_3}\big|f_2^{\ast}(\alpha)^2\mathcal{E}(\alpha)^2\big|d\alpha
\ \ll\ &
P_4^{\varepsilon}Z^2+P_4^{3/2+\tau+\varepsilon}\psi(X)^{2}Z,\end{align*}
whence
\begin{align*}\mathcal{K} \ll P_4^{11/4+\tau/2+\varepsilon}\psi(X)Z
+P_4^{7/2+\tau+\varepsilon}\psi(X)^{2}Z^{1/2}.\end{align*}When
$\psi(X)\ll X^{1/64-\tau}$, one has
\begin{align*}\mathcal{K}\ \ll
\ P_4^{4-\frac{1}{4}+\varepsilon}Z^{1/2}+
P_4^{3-\tau+\varepsilon}\psi(X)^{-1}Z.\end{align*} We conclude
from above that
\begin{align}\label{K}\mathcal{K}\ \ll \ P_4^{4-\frac{1}{4}+\frac{s-3}{2}+\varepsilon}Z^{1/2}+
P_4^{s-\tau+\varepsilon}\psi(X)^{-1}Z.\end{align}
 By
(\ref{I3K}) and (\ref{K}), we obtain
\begin{align}\label{I3}
\mathcal{I}_3 \ll
P_4^{4-\frac{1}{4}+\frac{s-3}{2}+\varepsilon}Z^{1/2}+P_4^{s-\tau+\varepsilon}\psi(X)^{-1}Z.\end{align}
We complete the proof.\end{proof}

\begin{lemma}\label{lemma34}Let $\mathcal{I}_4$ be defined in
(\ref{Ij}). Then we have
\begin{align*}\mathcal{I}_4 \ll & \ ZP_4^{s-(s-2)\nu/4+\varepsilon}.\end{align*}
\end{lemma}\begin{proof}
In view of (\ref{difference}) and (\ref{f-ast}), for
$\alpha\in\mathfrak{M}_{P_4}(q,a)$, one has
\begin{align*}f_4(\alpha)\ \ll\ & P_4w_4(q)(1+X|\alpha-a/q|)^{-1/4}+P_4^{1/2+\varepsilon}
\\ \  \ll \ & P_4^{1+\varepsilon}w_4(q)(1+X|\alpha-a/q|)^{-1/4},\end{align*}
and  \begin{align*}f_2(\alpha)\ \ll\ &
P_2q^{-1/2}(1+X|\alpha-a/q|)^{-1/2}.\end{align*}Therefore we
obtain
\begin{align*}\mathcal{I}_4 \ll & \ Z\sup_{\alpha\in\mathfrak{m}_4}|f_4(\alpha)|^{s-2}
\int_{\mathfrak{M}(P_4)}|f_4(\alpha)f_2(\alpha)|^2d\alpha
\\ \ll &\ ZP_4^{(s-2)(1-\nu/4)+\varepsilon}P_4^2P_2^2\sum_{q\le P_4}w_4(q)^2
\int_{|\beta|\le \frac{P_4}{qX}}(1+X|\beta|)^{-3/2} d\beta
\\ \ll &\ ZP_4^{2+(s-2)(1-\nu/4)+\varepsilon}\sum_{q\le P_4}w_4(q)^2.\end{align*}
In light of Lemma 2.4 by Kawada and Wooley \cite{KW}, one can
conclude that
\begin{align}\label{I4}
\mathcal{I}_4 \ll & \ ZP_4^{2+(s-2)(1-\nu/4)+\varepsilon}\ll
ZP_4^{s-(s-2)\nu/4+\varepsilon}.\end{align} The desired estimate
is established.\end{proof}

\vskip6mm

\begin{proof}[Proof of Theorem \ref{theorem}]
We denote by $Z_s(X)$ the set of integers $n$ with $X/2< n\le X$
for which the lower bound \begin{align*}\big|R_s(n)-c_s
\Gamma(\frac{5}{4})^4\mathfrak{S}_s(n)n^{s/4}\big|>n^{s/4}\psi(n)^{-1}\end{align*}holds,
and we abbreviate card($Z_s(X)$) to $Z_s$. It follows from
(\ref{orth}) and Lemma \ref{lemmaM} that for $n\in Z_s(X)$,
\begin{align*}\Big|\int_{\mathfrak{m}}f_2(\alpha)^2f_4(\alpha)^se(-n\alpha)d\alpha\Big|\gg X^{s/4}\psi(X)^{-1},\end{align*}
whence
\begin{align*}\sum_{n\in Z_s(X)}\Big|\int_{\mathfrak{m}}f_2(\alpha)^2f_4(\alpha)^se(-n\alpha)d\alpha\Big|\gg Z_sX^{s/4}\psi(X)^{-1}.\end{align*}
We choose complex numbers $\eta(n)$, with $|\eta(n)|=1$,
satisfying
\begin{align*}\Big|\int_{\mathfrak{m}}f_2(\alpha)^2f_4(\alpha)^se(-n\alpha)d\alpha\Big|=
\eta(n)\int_{\mathfrak{m}}f_2(\alpha)^2f_4(\alpha)^se(-n\alpha)d\alpha.\end{align*}
Then we define the exponential sum $\mathcal{E}_s(\alpha)$ by
\begin{align*}\mathcal{E}_s(\alpha)=\sum_{n\in Z_s(X)}
\eta(n)e(-n\alpha).\end{align*}One finds that
\begin{align}\label{Z}
Z_sX^{s/4}\psi(X)^{-1}\ll
\int_{\mathfrak{m}}\big|f_2(\alpha)^2f_4(\alpha)^s\mathcal{E}_s(\alpha)\big|d\alpha.\end{align}
Note that $\mathfrak{m}=\mathfrak{m}_1\cup \mathfrak{m}_2\cup
\mathfrak{m}_3\cup \mathfrak{m}_4$. Now we conclude from Lemmata
\ref{lemma31}-\ref{lemma34} and (\ref{Z}) that
\begin{align*}
Z_sX^{s/4}\psi(X)^{-1}\ll
P_4^{4-\frac{1}{4}+\frac{s-3}{2}+\varepsilon}Z_s^{1/2}+P_4^{s-\delta}\psi(X)^{-1}Z_s
\end{align*}for some sufficiently small positive number $\delta$. Therefore we have \begin{align}\label{fin}
Z_sX^{s/4}\psi(X)^{-1} \ll
X^{1-\frac{1}{16}+\frac{s-3}{8}+\varepsilon}Z_s^{1/2}.\end{align}
The estimate (\ref{fin}) implies $Z_3\ll
X^{3/8+\varepsilon}\psi(X)^{2}$ and $Z_4\ll
X^{1/8+\varepsilon}\psi(X)^{2}$. The proof of Theorem
\ref{theorem} is completed by summing over dyadic
intervals.\end{proof}

\vskip3mm

\subsection*{Acknowledgements} The author would like to thank the
referees for many helpful comments and suggestions.

\end{document}